\documentclass[12pt,final]{amsart}

\usepackage{amsmath}
\usepackage{amsthm}
\usepackage{amssymb}
\usepackage{amsfonts}
\usepackage{enumerate}
\usepackage{color}
\usepackage[bookmarks=true,hyperindex,pdftex,colorlinks, citecolor=blue]{hyperref}

\makeatletter
\@namedef{subjclassname@2010}{%
  \textup{2010} Mathematics Subject Classification}
\makeatother

\DeclareMathOperator{\Lip}{Lip}
\DeclareMathOperator{\lip}{lip}
\DeclareMathOperator{\Ext}{Ext}
\DeclareMathOperator{\diam}{diam}
\DeclareMathOperator{\dist}{dist}
\DeclareMathOperator{\spn}{span}

\newcommand{\weaks}{\textit{w}$^\ast$}

\newcommand{\abs}[1]{\left|{#1}\right|}
\newcommand{\set}[1]{\left\{{#1}\right\}}
\newcommand{\norm}[1]{\left\|{#1}\right\|}
\newcommand{\dual}[1]{{#1}^\ast}
\newcommand{\ddual}[1]{{#1}^{\ast\ast}}
\newcommand{\ball}[1]{B_{{#1}}}
\newcommand{\duality}[1]{\left<{#1}\right>}
\newcommand{\metex}[1]{\varepsilon({#1})}
\newcommand{\lipfree}[1]{\mathcal{F}({#1})}
\newcommand{\setdpq}[1]{D_{#1}}

\newtheorem{theorem}{Theorem}[section]
\newtheorem*{maintheorem*}{Main Theorem}
\newtheorem{lemma}[theorem]{Lemma}
\newtheorem{corollary}[theorem]{Corollary}
\newtheorem{proposition}[theorem]{Proposition}
\newtheorem*{claim}{Claim}

\theoremstyle{definition}
\newtheorem*{definition*}{Definition}
\newtheorem{definition}[theorem]{Definition}
\newtheorem{example}[theorem]{Example}
\newtheorem{question}{Question}

\begin{document}

\title[The preserved extremal structure of $\lipfree{X}$]{On the preserved extremal structure of Lipschitz-free spaces}

\author[R. J. Aliaga]{Ram\'on J. Aliaga}
\address[R. J. Aliaga]{Instituto de F\'isica Corpuscular (CSIC-UV)\\
C/ Catedr\'atico Jos\'e Beltr\'an 2, 46980 Paterna, Spain}
\address{Instituto Universitario de Matem\'atica Pura y Aplicada\\
Universitat Polit\`ecnica de Val\`encia\\
Camino de Vera S/N, 46022 Valencia, Spain}
\email{raalva@upvnet.upv.es}

\author[A. J. Guirao]{Antonio J. Guirao}
\address[A. J. Guirao]{Instituto Universitario de Matem\'atica Pura y Aplicada\\
Universitat Polit\`ecnica de Val\`encia\\
Camino de Vera S/N, 46022 Valencia, Spain}
\email{anguisa2@mat.upv.es}

\date{} 


\begin{abstract}
We characterize preserved extreme points of Lipschitz-free spaces $\lipfree{X}$ in terms of simple geometric conditions on the underlying metric space $(X,d)$. Namely, each preserved extreme point corresponds to a pair of points $p,q$ in $X$ such that the triangle inequality $d(p,q)\leq d(p,r)+d(q,r)$ is uniformly strict for $r$ away from $p,q$. For compact $X$, this condition reduces to the triangle inequality being strict. This result gives an affirmative answer to a conjecture of N. Weaver that compact spaces are concave if and only if they have no triple of metrically aligned points.
\end{abstract}




\maketitle


\section{Introduction}

Given a pointed metric space $(X,d)$, i.e. one that has a designated base point $e$, the space $\Lip(X)$ of scalar valued Lipschitz functions on $X$ has a distinguished subspace $\Lip_0(X)$ consisting of those elements of $\Lip(X)$ that vanish at $e$. $\Lip_0(X)$ is then a Banach space endowed with the norm $L(f)$ given by the tightest Lipschitz constant of $f$, and different choices of base points lead to linearly isometric Banach spaces via the map $f\mapsto f-f(e)$.

It is well-known that $\Lip_0(X)$ is a dual space, and its canonical predual is $\lipfree{X}=\overline{\spn}\,j(X)\subset\dual{\Lip_0(X)}$, where $j\colon X\rightarrow\dual{\Lip_0(X)}$ maps each $x\in X$ to its evaluation operator $j(x)\colon f\mapsto f(x)$. Following \cite{godefroy_kalton_2003}, we call $\lipfree{X}$ the \emph{Lipschitz-free space over $X$}. Note that $j$ is a (non-linear) isometric embedding of $X$ into a linearly dense subset of $\lipfree{X}$ and, in fact, this is a universal property of $\lipfree{X}$: every non-expansive map from $X$ into a Banach space that maps $e$ to $0$ can be factored through $j$ \cite[Th. 2.2.4]{weaver}. For a recent survey on Lipschitz-free Banach spaces see \cite{godefroy_2015}.

The extremal structure of the unit ball of $\lipfree{X}$ reveals important details about the geometry of $X$. Of particular interest are the \emph{preserved extreme points}, i.e. those points of $\Ext(\ball{\lipfree{X}})$ that are also extreme points of $\ball{\ddual{\lipfree{X}}}=\ball{\dual{\Lip_0(X)}}$. For instance, their properties are used in \cite[Sections 2.6 and 2.7]{weaver} to obtain metric versions of the Banach-Stone theorem for $\Lip$ and $\Lip_0$ spaces under various hypotheses. Further information about preserved and unpreserved extreme points can be found in the recent survey \cite{GMZ14}.

When $X$ is complete, any preserved extreme points of $\ball{\lipfree{X}}$ are necessarily of the form
\begin{equation*}
u_{pq}:=\frac{j(p)-j(q)}{\norm{j(p)-j(q)}}=\frac{j(p)-j(q)}{d(p,q)}
\end{equation*}
for distinct $p,q\in X$ \cite[Cor. 2.5.4]{weaver}; the completeness of $X$ is crucial for this. In this paper, we study the geometric conditions under which these elements of $B_{\lipfree{X}}$ are indeed preserved extreme points. They can be stated in a simple form if we allow an abuse of notation and extend the metric function $d$ from pairs of points in $X$ to its Stone-\v{C}ech compactification $\beta X$. Our main result is the following:

\begin{maintheorem*}[cf. Theorem \ref{tm:pep_fx}]
If $X$ is a complete pointed metric space, then the preserved extreme points of $\ball{\lipfree{X}}$ are precisely the elements $u_{pq}$ where $p,q$ are distinct points of $X$ such that $d(p,q)<d(p,r)+d(q,r)$ for all $r\in\beta X\setminus\set{p,q}$.
\end{maintheorem*}

In terms of the geometry within $X$, this characterization is equivalent to the triangle inequality being uniformly strict for $r$ away from $p$ and $q$; the precise statement is given in Lemma \ref{lm:pq_concave}.

As a consequence of this result, in Corollary \ref{co:concave} we solve in the positive a conjecture of N. Weaver stating that compact spaces such that $d(p,q)<d(p,r)+d(r,q)$ for any triple of distinct points $p,q,r$ are concave \cite[Open Problem in p. 53]{weaver}. Another implication is that all extreme points of $\ball{\lipfree{X}}$ of the form $u_{pq}$ are preserved when $X$ is compact (Theorem \ref{tm:pep_fk}).

Moreover, we also find a sufficient condition for $u_{pq}$ to be a preserved extreme point (Proposition \ref{pr:local_peaking}) that improves the well-known \cite[Prop. 2.4.2]{weaver}, replacing the single, globally peaking function with a family of functions that peak locally. Example~\ref{ex:ap_aq} implies that neither of these results are characterizations, even in the compact case.

Throughout the paper, $X$ will denote a metric space with metric $d$; if $X$ is pointed, its base point will be denoted by $e$. We will use standard notation: $\ball{Y}$ for the closed unit ball of normed space $Y$, and $\duality{\dual{x},x}$ for the evaluation of the functional $\dual{x}$ at the point $x$. We will also restrict ourselves to the case of real scalars. The main reason is that this supports the following metric version of Tietze's extension theorem, which is not valid in the complex case \cite[p. 18]{weaver}.
\begin{proposition}[{\cite[Th. 1.5.6]{weaver}}]
\label{pr:metric_tietze}
Let $X$ be a metric space and $Y\subset X$. Then every $f\colon Y\rightarrow\mathbb{R}$ can be extended to $X$ in such a way that $L(f)$ and $\norm{f}_\infty$ are preserved.
\end{proposition}

For the non-defined notions used through this article, we refer to \cite{LosCinco}.

\section{Metric alignment and extremal structure}

\begin{definition}
Let $X$ be a metric space and $p,q,r\in X$. We say that \emph{$r$ lies between $p$ and $q$} if $d(p,r)+d(r,q)=d(p,q)$; if $r$ is neither $p$ nor $q$, we say that it lies \emph{strictly} between $p$ and $q$. We also say that three distinct points of $X$ are \emph{metrically aligned} if one of them lies strictly between the other two. The \emph{metric segment} $[p,q]$ is defined as the set of all points of $X$ that lie between $p$ and $q$.
\end{definition}

Observe that this definition of metric alignment coincides with the intuitive notion of alignment in the Euclidean plane or space. More generally, if $X$ is a subset of a strictly convex normed space, then $p,q,r$ are metrically aligned if and only if they are \emph{linearly aligned}, i.e. if they span an affine subspace of dimension 1 instead of 2, or equivalently, if $p-r$ and $q-r$ are linearly dependent.

We also introduce the notation
\begin{equation*}
\metex{r;p,q}:=d(r,p)+d(r,q)-d(p,q) \text{.}
\end{equation*}
Note that $\metex{r;p,q}\geq 0$, and $\metex{r;p,q}=0$ if and only if $r$ lies between $p$ and $q$. Note also that $[p,q]$ is closed, always contains $p$ and $q$, and it is possible for it to contain no other point. Finally, note that $\metex{x;p,q}\leq 2\dist(x,[p,q])$ for any $x\in X$; this is proven by adding the triangle inequalities $d(p,x)\leq d(p,r)+d(r,x)$ and $d(q,x)\leq d(q,r)+d(r,x)$ for $r\in [p,q]$.

Since the mapping $r\mapsto\metex{r;p,q}$ is continuous in $X$, it can be extended continuously to a mapping $\beta X\rightarrow[0,\infty]$, where $\beta X$ is the Stone-\v{C}ech compactification of $X$. Thus, for $\xi\in\beta X$, we will denote by $\metex{\xi;p,q}$ the result of applying that mapping to $\xi$, i.e. $\metex{\xi;p,q}=\lim_i\metex{x_i;p,q}$ if $\set{x_i:i\in I}$ is a net in $X$ that converges to $\xi$. We will then say that $\xi$ \emph{lies strictly between $p$ and $q$} if $\metex{\xi;p,q}=0$ and $\xi$ is neither $p$ nor $q$.

There is a strong relationship between metric alignment in $X$ and the extremal structure of $\ball{\lipfree{X}}$, as illustrated by the following result:

\begin{proposition}
\label{pr:aligned_combination}
Let $X$ be a pointed metric space and $p,q$ distinct points of $X$.
\begin{enumerate}[\upshape (a)]
\item If $u_{pq}$ is an extreme point of $\ball{\lipfree{X}}$, then no point of $X$ lies strictly between $p$ and $q$.
\item If $u_{pq}$ is a preserved extreme point of $\ball{\lipfree{X}}$, then no point of $\beta X$ lies strictly between $p$ and $q$.
\end{enumerate}
\end{proposition}

\begin{proof}
(a) For any $r\in X\setminus\set{p,q}$ we have
\begin{equation*}
u_{pq}=\frac{j(p)-j(q)}{d(p,q)}=\frac{j(p)-j(r)}{d(p,q)}+\frac{j(r)-j(q)}{d(p,q)}=\frac{d(p,r)}{d(p,q)}\,u_{pr}+\frac{d(r,q)}{d(p,q)}\,u_{rq} \text{.}
\end{equation*}
If $d(p,q)=d(p,r)+d(q,r)$, then this expresses $u_{pq}$ as a convex combination of elements $u_{pr}$ and $u_{rq}$ of $\ball{\lipfree{X}}$ so it cannot be an extreme point.

(b) Suppose that $\metex{\xi;p,q}=0$ for some $\xi\in\beta X\setminus\set{p,q}$. Let $\set{x_i:i\in I}$ be a net in $X$ that converges to $\xi$. We may assume that $\metex{x_i;p,q}$ is bounded, hence so are $d(x_i,p)$ and $d(x_i,q)$. Thus the limits $d(p,\xi)=\lim_i d(p,x_i)$ and $d(\xi,q)=\lim_i d(x_i,q)$ exist and are finite and positive; moreover, $d(p,\xi)+d(\xi,q)=\metex{\xi;p,q}+d(p,q)=d(p,q)$.

Let $V$ be a closed ball with center in $e$ and a radius large enough to contain $p$, $q$ and all the $x_i$. The restricted operator $j|_V\colon V\rightarrow\dual{\Lip_0(X)}$ is \weaks-continuous and its range is contained in $\diam(V)\cdot\ball{\dual{\Lip_0(X)}}$ which is \weaks-compact. Hence $j|_V$ can be extended \weaks-continuously to $\beta V$, and in particular there is $\Lambda=j|_V(\xi)\in\dual{\Lip_0(X)}$ such that $\duality{\Lambda,f}=\lim_i f(x_i)$ for $f\in\Lip_0(X)$. For any $f\in\ball{\Lip_0(X)}$ we have
\begin{equation*}
\abs{\duality{j(p)-\Lambda,f}}=\lim_i\abs{f(p)-f(x_i)}\leq\lim_i d(p,x_i)=d(p,\xi)
\end{equation*}
and so $u_{p\xi}:=(j(p)-\Lambda)/d(p,\xi)$ is an element of $\ball{\dual{\Lip_0(X)}}$. Analogously, $u_{\xi q}:=(\Lambda-j(q))/d(\xi,q)\in\ball{\dual{\Lip_0(X)}}$. Then we can express
\begin{equation*}
u_{pq}=\frac{j(p)-j(q)}{d(p,q)}=\frac{j(p)-\Lambda}{d(p,q)}+\frac{\Lambda-j(q)}{d(p,q)}=\frac{d(p,\xi)}{d(p,q)}\,u_{p\xi}+\frac{d(\xi,q)}{d(p,q)}\,u_{\xi q} \text{,}
\end{equation*}
i.e. $u_{pq}$ is a convex combination of elements in $\ball{\dual{\Lip_0(X)}}$, so it cannot be a preserved extreme point.
\end{proof}

The condition in Proposition \ref{pr:aligned_combination}(b) essentially says that it is not possible to have $d(p,r_i)+d(q,r_i)\rightarrow d(p,q)$ unless $\set{r_i}$ clusters at $p$ or $q$. Equivalently, the triangle inequality is uniformly strict for $r$ away from $p,q$. The precise formulation is the following:

\begin{lemma}
\label{lm:pq_concave}
Let $X$ be a metric space and $p,q$ distinct points of $X$. Then the following are equivalent:
\begin{enumerate}[\upshape (i)]
\item no point of $\beta X$ lies strictly between $p$ and $q$,
\item for every $\varepsilon>0$ there is $\delta>0$ such that $\metex{r;p,q}\geq\delta$ whenever $r\in X$ satisfies $d(p,r)\geq\varepsilon$ and $d(q,r)\geq\varepsilon$.
\end{enumerate}
\end{lemma}

\begin{proof}
Suppose (i) is false and there is $\xi\in\beta X\setminus\set{p,q}$ such that $\metex{\xi;p,q}=0$. Then there is a net $\set{x_i:i\in I}$ in $X$ such that $x_i\rightarrow\xi$ and $\metex{x_i;p,q}\rightarrow 0$. Choose $\varepsilon>0$ such that $d(x_i,p)>\varepsilon$ and $d(x_i,q)>\varepsilon$ eventually; such an $\varepsilon$ exists because $\set{x_i}$ would otherwise have a subsequence that converges to $p$ or $q$. Then (ii) is false for this $\varepsilon$.

Suppose now that (ii) is false, and choose $\varepsilon>0$ such that for every $n\in\mathbb{N}$ there is $r_n\in X$ such that $d(p,r_n)\geq\varepsilon$, $d(q,r_n)\geq\varepsilon$ and $\metex{r_n;p,q}<2^{-n}$. Let $\xi$ be a cluster point of $r_n$ in $\beta X$. Then clearly $\xi$ lies strictly between $p$ and $q$, so (i) is false.
\end{proof}

\section{Norm attainment of Lipschitz functions}

We borrow the following notation from \cite[Chapter 2]{weaver}: denote
\begin{equation*}
\widetilde{X}:=\set{(x,y):x,y\in X,x\neq y}
\end{equation*}
with the subspace topology of $X^2$, and define the map $\Phi\colon\Lip(X)\rightarrow C(\widetilde{X})$ by
\begin{equation*}
\Phi f(p,q):=\frac{f(p)-f(q)}{d(p,q)}=\duality{u_{pq},f}, {\,\,\text{ for }f\in\Lip(X)\text{ and } (p,q)\in\widetilde{X}.}
\end{equation*}
Note that $L(f)=\norm{\Phi f}_{\infty}$, so $\Phi$ is in fact a linear isometry from $\Lip_0(X)$ into $\ell^{\infty}(\widetilde{X})$. Moreover, since the function $\Phi f\in C(\widetilde{X})$ is bounded by $L(f)$, it can be extended continuously to $\beta\widetilde{X}$, the Stone-\v{C}ech compactification of $\widetilde{X}$; hence $\Phi f$ can be identified with an element in $C(\beta\widetilde{X})$, and $\Phi$ can be regarded as a map from $\Lip_0(X)$ into $C(\beta\widetilde{X})$. For arbitrary $\zeta\in\beta\widetilde{X}$, we will write $\Phi f(\zeta)$ to refer to the value at $\zeta$ of the extension of $\Phi f$; equivalently, $\Phi f(\zeta)=\lim_i \Phi f(x_i,y_i)$ if {$\set{(x_i,y_i):i\in I}$ is a net converging to $\zeta$} in $\beta\widetilde{X}$. Recall that the dual of $C(\beta\widetilde{X})$ is $M(\beta\widetilde{X})$, the space of real regular Borel measures on $\beta\widetilde{X}$, so that for each $\dual{x}\in\dual{\Lip_0(X)}$ there is a measure $\mu\in M(\beta\widetilde{X})$ of equal norm such that $\dual{\Phi}\mu=\dual{x}$, where $\dual{\Phi}\colon M(\beta\widetilde{X})\rightarrow\dual{\Lip_0(X)}$ is the adjoint operator of $\Phi$.

\begin{definition}
\label{def:norm_attainment}
Let $f\in\Lip(X)$, $f\neq 0$ and $\zeta\in\beta\widetilde{X}$. We say that $f$ \emph{attains its (Lipschitz) norm at $\zeta$} if $\abs{\Phi f(\zeta)}=L(f)$. We say that $f$ \emph{peaks at $(p,q)\in\widetilde{X}$} if it attains its norm at $(p,q)$ and, for every open $U\subset\widetilde{X}$ containing $(p,q)$ and $(q,p)$, there is $c<L(f)$ such that ${\abs{\Phi f(x,y)}\leq c}$ for all $(x,y)\in\widetilde{X}\setminus U$.
\end{definition}

Informally, $f$ peaks at $(p,q)$ if $\abs{\Phi f}$ is uniformly less than $L(f)$ away from $(p,q)$ and $(q,p)$. This is a strong condition, and it is a well-known result that it is sufficient to ensure the existence of preserved extreme points:

\begin{proposition}[{\cite[Prop. 2.4.2]{weaver}}]
\label{pr:weaver_peaking}
Let $X$ be a pointed metric space and suppose that there is a function in $\Lip_0(X)$ that peaks at $(p,q)\in\widetilde{X}$. Then $u_{pq}$ is a preserved extreme point of $\ball{\lipfree{X}}$.
\end{proposition}

We wish to generalize this result by finding weaker sufficient conditions. In order to do this, for a given $(p,q)\in\widetilde{X}$ we will consider the set
\begin{align*}
\setdpq{(p,q)}:=\big\{\zeta\in\beta\widetilde{X}: \text{if $f\in\Lip_0(X)$ attains its norm at $(p,q)$,} & \\
\text{then $f$ also attains its norm at $\zeta$} & \big\} \text{.}
\end{align*}
Notice that $\setdpq{(p,q)}$ is closed, hence compact. Notice also that $(p,q)$ and $(q,p)$ are always in $\setdpq{(p,q)}$. It is possible for $\setdpq{(p,q)}$ to contain no other points beside these two;  this happens, for instance, when there is $f\in\Lip_0(X)$ that peaks at $(p,q)$, as that same $f$ shows that every other $\zeta\in\beta\widetilde{X}$ fails to fulfill the condition in the definition. A refinement of the argument used in the proof of {\cite[Prop. 2.4.2]{weaver}} yields the following:

\begin{lemma}
\label{lm:measure_support}
Let $X$ be a pointed metric space and $(p,q)\in\widetilde{X}$. Suppose that $u_{pq}=\lambda\dual{x}_1+(1-\lambda)\dual{x}_2$ for some $\lambda\in(0,1)$ and $\dual{x}_1,\dual{x}_2\in\ball{\dual{\Lip_0(X)}}$. Then there are $\mu_1,\mu_2\in\ball{M(\beta\widetilde{X})}$ concentrated on $\setdpq{(p,q)}$ such that $\dual{x}_1=\dual{\Phi}\mu_1$ and $\dual{x}_2=\dual{\Phi}\mu_2$.
\end{lemma}

\begin{proof}
Take measures $\mu_i\in\ball{M(\beta\widetilde{X})}$ such that $\Phi^{\ast}\mu_i=\dual{x}_i$ for $i=1,2$. Notice that, for any $f\in\ball{\Lip_0(X)}$ such that $\Phi f(p,q)=1$, the inequalities
\begin{align*}
1=\Phi f(p,q)=\duality{u_{pq},f} &= \lambda\duality{\dual{x}_1,f}+(1-\lambda)\duality{\dual{x}_2,f} \\
&= \lambda\duality{\mu_1,\Phi f}+(1-\lambda)\duality{\mu_2,\Phi f} \\
&\leq \lambda\norm{\mu_1}\norm{\Phi f}_{\infty}+(1-\lambda)\norm{\mu_2}\norm{\Phi f}_{\infty}\leq 1
\end{align*}
hold and so we must have $\duality{\mu_1,\Phi f}=\duality{\mu_2,\Phi f}=1$. Now fix $\mu\in\set{\mu_1,\mu_2}$; we will show that $\mu$ is concentrated on $\setdpq{(p,q)}$.

Let $\zeta\in\beta\widetilde{X}\setminus\setdpq{(p,q)}$. Then there is $f\in\Lip_0(X)$ that attains its norm at $(p,q)$ but not at $\zeta$. We may assume that $\Phi f(p,q)=L(f)=1$ and $\abs{\Phi f(\zeta)}<1$. Since $\Phi f$ is continuous, there are $c\in (0,1)$ and an open neighborhood $U(\zeta)\subset\beta\widetilde{X}$ of $\zeta$ such that $\abs{\Phi f(\zeta^\prime)}\leq c$ for every $\zeta^\prime\in U(\zeta)$. But then
\begin{align*}
1&=\int_{\beta\widetilde{X}}(\Phi f)\,d\mu=\int_{U(\zeta)}(\Phi f)\,d\mu+\int_{\beta\widetilde{X}\setminus U(\zeta)}(\Phi f)\,d\mu \\
&\leq c\abs{\mu}(U(\zeta))+\abs{\mu}(\beta\widetilde{X}\setminus U(\zeta))\leq 1-(1-c)\abs{\mu}(U(\zeta))
\end{align*}
where $\abs{\mu}$ is the total variation of $\mu$. Since $c<1$ we obtain $\abs{\mu}(U(\zeta))=0$.

Now let $K$ be any compact subset of $\beta\widetilde{X}\setminus\setdpq{(p,q)}$. Then $\set{U(\zeta):\zeta\in K}$ is an open cover of $K$ so it admits a finite subcover $K\subset\bigcup_{j=1}^n U(\zeta_j)$, hence
\begin{equation*}
\abs{\mu}(K)\leq\sum_{j=1}^n\abs{\mu}(U(\zeta_j))=0 \text{.}
\end{equation*}
Since $\abs{\mu}$ is regular and $\beta\widetilde{X}\setminus\setdpq{(p,q)}$ is open, $\abs{\mu}(\beta\widetilde{X}\setminus\setdpq{(p,q)})$ is the supremum of such $\abs{\mu}(K)$, which implies that it is equal to zero. It follows that $\mu$ is concentrated on $\setdpq{(p,q)}$.
\end{proof}

As a consequence, the peaking function in Proposition \ref{pr:weaver_peaking} can be replaced by a family of norm attaining functions $f$ such that the regions where $\abs{\Phi f}<L(f)$ cover all of $\beta\widetilde{X}$ except for $(p,q)$ and $(q,p)$. This is equivalent to saying that $\setdpq{(p,q)}=\set{(p,q),(q,p)}$.

\begin{proposition}
\label{pr:local_peaking}
Let $X$ be a pointed metric space and $(p,q)\in\widetilde{X}$ such that, for any $\zeta\in\beta\widetilde{X}\setminus\set{(p,q),(q,p)}$, there is $f\in\Lip_0(X)$ such that $\Phi f(p,q)=L(f)$ and $\abs{\Phi f(\zeta)}<L(f)$. Then $u_{pq}$ is a preserved extreme point of $\ball{\lipfree{X}}$.
\end{proposition}

\begin{proof}
Suppose that $u_{pq}=\lambda\dual{x}_1+(1-\lambda)\dual{x}_2$ for some $\lambda\in(0,1)$ and $\dual{x}_1,\dual{x}_2\in\ball{\dual{\Lip_0(X)}}$. By Lemma \ref{lm:measure_support}, for $i=1,2$ we have $\dual{x}_i=\dual{\Phi}\mu_i$ where $\mu_i\in\ball{M(\beta\widetilde{X})}$ is concentrated on $\setdpq{(p,q)}=\set{(p,q),(q,p)}$. But the Dirac measure $\delta_{(p,q)}$ on $(p,q)$ satisfies
\begin{equation*}
\duality{\Phi^{\ast}\delta_{(p,q)},f}=\Phi f(p,q)=\duality{u_{pq},f}
\end{equation*}
for any $f\in\Lip_0(X)$, so $\Phi^{\ast}\delta_{(p,q)}=u_{pq}$. For $i=1,2$, $\dual{x}_i$ is therefore a linear combination of $u_{pq}$ and $u_{qp}=-u_{pq}$ and it follows that $\dual{x}_i=u_{pq}$. Hence $u_{pq}\in\Ext(\ball{\dual{\Lip_0(X)}})$ as was to be shown.
\end{proof}

Next, we show that the elements of $\setdpq{(p,q)}$ have a very specific form when $p$ and $q$ satisfy the condition in Lemma \ref{lm:pq_concave}. Let $p\in X$ and $\zeta\in\beta\widetilde{X}$. Following \cite[Section 4.7]{weaver}, we say that $\zeta$ \emph{lies over $p$} if it is the limit of a net $\set{(x_i,y_i):i\in I}$ in $\widetilde{X}$ such that $\lim_i x_i=\lim_i y_i=p$. Notice that if $p$ is an isolated point in $X$ then no point of $\beta\widetilde{X}$ can lie over $p$. Notice also that if $X$ is compact, then each point of $\beta\widetilde{X}$ either belongs to $\widetilde{X}$ or lies over some $p\in X$.

\begin{proposition}
\label{pr:dpq_concave}
Let $X$ be a pointed metric space and $p,q$ distinct points of $X$. Suppose that no point of $\beta X$ lies strictly between $p$ and $q$. Then $\setdpq{(p,q)}=\set{(p,q),(q,p)}\cup A_p\cup A_q$ where $A_p$ (resp. $A_q$) consists of points of $\beta\widetilde{X}$ that lie over $p$ (resp. $q$).
\end{proposition}

\begin{proof}
Let $\zeta\in\beta\widetilde{X}$, and let $\set{(x_i,y_i):i\in I}$ be a net in $\widetilde{X}$ that converges to $\zeta$ in $\beta\widetilde{X}$. Choose a subnet such that $\set{x_i}$ and $\set{y_i}$ converge to elements $\xi$ and $\eta$ in $\beta X$; call that subnet $(x_i,y_i)$ again. First we prove the following claim:

\begin{claim}
If $\set{x_i}$ does not converge to $p$ or $q$, then it has a subnet $\set{x_j}$ such that 
\begin{equation}
\lim_j \frac{\metex{x_j;p,q}}{d(x_j,q)}>0 \text{.}
\label{eq:claim_lim}
\end{equation}
\end{claim}

\begin{proof}[Proof of the claim]
Take $\varepsilon>0$ such that $d(x_i,p)>\varepsilon$ and $d(x_i,q)>\varepsilon$ eventually. By Lemma \ref{lm:pq_concave}, there is $\delta>0$ such that $\metex{x_i;p,q}\geq\delta$ eventually. Hence, if $d(x_i,q)$ is eventually bounded by $M<\infty$, then the limit \eqref{eq:claim_lim} is at least $\delta/M>0$. Otherwise, choose a subnet $\set{x_j}$ such that $d(x_j,q)\rightarrow\infty$. Then also $d(x_j,p)\geq d(x_j,q)-d(p,q)\rightarrow\infty$, and
\begin{equation*}
\lim_j \frac{d(x_j,p)}{d(x_j,q)}\leq\lim_j \frac{d(x_j,q)+d(q,p)}{d(x_j,q)}=1+\lim_j\frac{d(p,q)}{d(x_j,q)}=1
\end{equation*}
and by symmetry in $p$ and $q$ we get $\lim_j d(x_j,p)/d(x_j,q)=1$, hence
\begin{equation*}
\lim_j \frac{\metex{x_j;p,q}}{d(x_j,q)}=1+\lim_j\frac{d(x_j,p)-d(p,q)}{d(x_j,q)}=2
\end{equation*}
is positive.
\end{proof}

We need to show that $\zeta\in\setdpq{(p,q)}$ implies that $\xi,\eta\in\set{p,q}$. We will assume otherwise, and construct $f\in\ball{\Lip_0(X)}$ such that $\Phi f(p,q)=1$ and $\abs{\Phi f(\zeta)}<1$.

Suppose first that $\xi,\eta\in\beta X\setminus\set{p,q}$. By the claim, we can replace $\set{(x_i,y_i)}$ with a subnet such that
\begin{equation*}
c=\min\set{1,\inf_i\frac{\metex{x_i;p,q}}{d(x_i,q)}, \inf_i\frac{\metex{y_i;p,q}}{d(y_i,q)}}>0 \text{.}
\end{equation*}
Let $Z=\set{p,q}\cup\bigcup_{i\in I}\set{x_i,y_i}$, choose $\alpha\in (0,c)$ and define $g\colon Z\rightarrow\mathbb{R}$ by
\begin{equation*}
g(x)=\begin{cases}
d(p,q) & \text{if } x=p \\
(1-\alpha)\cdot d(x,q) & \text{if } x\in Z\setminus\set{p} \text{.}
\end{cases}
\end{equation*}
It is clear that $\Phi g(p,q)=1$ and $\abs{\Phi g(x,y)}\leq 1-\alpha$ for $x,y\in Z\setminus\set{p}$. For any $x\in Z\setminus\set{p,q}$ we have
\begin{align*}
1-\Phi g(p,x) &= \frac{\metex{x;p,q}-\alpha d(x,q)}{d(p,x)} \geq (c-\alpha)\frac{d(x,q)}{d(p,x)}>0 \\
1+\Phi g(p,x) &= \frac{\metex{p;x,q}+\alpha d(x,q)}{d(p,x)} \geq \alpha\,\frac{d(x,q)}{d(p,x)}>0
\end{align*}
so $-1<\Phi g(p,x)<1$, hence $L(g)=1$. Now extend $g$ from $Z$ to $X$ using Proposition \ref{pr:metric_tietze} and let $f=g-g(e)$. Then $f\in\ball{\Lip_0(X)}$, $\Phi f(p,q)=1$, and $\abs{\Phi f(\zeta)}=\lim_i\abs{\Phi f(x_i,y_i)}\leq 1-\alpha<1$, hence $\zeta\notin\setdpq{(p,q)}$.

Now suppose that exactly one of $\xi,\eta$ is in $\set{p,q}$; without loss of generality, assume that $\eta=q$. Then we can repeat the construction above with
\begin{equation*}
c=\min\set{1,\inf_i\frac{\metex{x_i;p,q}}{d(x_i,q)}}>0
\end{equation*}
and $Z=\set{p,q}\cup\bigcup_{i\in I}\set{x_i}$. Again we obtain $f\in\ball{\Lip_0(X)}$ such that $\Phi f(p,q)=1$, and $\abs{\Phi f(\zeta)}=\lim_i\abs{\Phi f(x_i,q)}\leq 1-\alpha<1$ so that $\zeta\notin\setdpq{(p,q)}$. This concludes the proof.
\end{proof}

If $A_p$ and $A_q$ are empty, we can apply Proposition \ref{pr:local_peaking} to conclude that $u_{pq}$ is a preserved extreme point of $\ball{\lipfree{X}}$. However this is not generally the case. The following technical lemma will be used in Example \ref{ex:ap_aq} to build compact spaces that have no triple of metrically aligned points and yet one or both of $A_p$, $A_q$ are nonempty; they will show that the condition in Proposition \ref{pr:local_peaking} is sufficient but not necessary for preserved extremality.

\begin{lemma}
\label{lm:dpq_nonempty}
Let $X$ be a pointed metric space and $p,q$ distinct points of $X$. Suppose that there is a sequence $\set{q_n}$ in $X\setminus\set{q}$ such that $q_n\rightarrow q$ and $\metex{q_n;p,q}/d(q_n,q)\rightarrow 0$. Then $\setdpq{(p,q)}$ contains a point that lies over $q$.
\end{lemma}

\begin{proof}
Since $d(q_n,q)\rightarrow 0$, we may assume that $d(q_n,q)$ is strictly decreasing and that $d(q_{n+1},q)/d(q_n,q)\rightarrow 0$ by selecting a subsequence. Then $\zeta_n=(q_n,q_{n+1})\in\widetilde{X}$ for every $n\in\mathbb{N}$ and, since $\beta\widetilde{X}$ is compact, the sequence $\set{\zeta_n}$ must have a subnet that converges to some $\zeta\in\beta\widetilde{X}$. Clearly $\zeta$ lies over $q$. We will show that $\zeta\in\setdpq{(p,q)}$.

Define $h\in\ball{\Lip_0(X)}$ by $h(x)=d(x,q)$ for $x\in X$. Then
\begin{equation*}
\Phi h(\zeta_n) = \frac{d(q_n,q)-d(q_{n+1},q)}{d(q_n,q_{n+1})} \geq \frac{d(q_n,q)-d(q_{n+1},q)}{d(q_n,q)+d(q_{n+1},q)} = \frac{1-\frac{d(q_{n+1},q)}{d(q_n,q)}}{1+\frac{d(q_{n+1},q)}{d(q_n,q)}}
\end{equation*}
and since $\Phi h\leq 1$ we obtain $\Phi h(\zeta_n)\rightarrow 1$ and thus $\Phi h(\zeta)=1$. Now let $f\in\ball{\Lip_0(X)}$ be such that $\Phi f(p,q)=1$. From $\Phi f(q_{n+1},q)\leq 1$ we obtain $f(q_{n+1})\leq f(q)+h(q_{n+1})$, and from $\Phi f(p,q_n)\leq 1$ we get
\begin{equation*}
f(q_n)\geq f(p)-d(p,q_n)=f(q)+h(q_n)-\metex{q_n;p,q} \text{.}
\end{equation*}
Subtracting both inequalities yields $f(q_n)-f(q_{n+1}) \geq h(q_n)-h(q_{n+1})-\metex{q_n;p,q}$, hence
\begin{equation*}
\Phi f(\zeta_n) \geq \Phi h(\zeta_n) - \frac{\metex{q_n;p,q}}{d(q_n,q_{n+1})} \geq \Phi h(\zeta_n) - \frac{\metex{q_n;p,q}}{d(q_n,q)}\frac{1}{1-\frac{d(q_{n+1},q)}{d(q_n,q)}}
\end{equation*}
and taking limits we get $\Phi f(\zeta)\geq \Phi h(\zeta)$, hence $\Phi f(\zeta)=1$.
\end{proof}

\begin{example}
\label{ex:ap_aq}
In $\mathbb{R}^2$, choose distinct points $p$ and $q$ at unit distance, and $\lambda\in(0,1)$. We construct sequences $\set{p_n}$ and $\set{q_n}$ iteratively as follows: let $q_0=p$ and $p_0=q$. Suppose that $p_0,\ldots,p_{n-1}$ and $q_0,\ldots,q_{n-1}$ have been chosen. Then take $p_n$ in the ball with center $p+\lambda^n(q-p)$ and radius $\lambda^{2n}$, such that $p_n$ is not aligned with any pair of points in $\set{p_0,\ldots,p_{n-1},q_0,\ldots,q_{n-1}}$; this is always possible because the ball has positive measure while the set of lines spanned by a finite amount of pairs of points is a null set. Similarly, take $q_n$ in the ball with center $q+\lambda^n(p-q)$ and radius $\lambda^{2n}$ but not aligned with any pair of points in $\set{p_0,\ldots,p_{n-1},p_n,q_0,\ldots,q_{n-1}}$.

The space $X=\set{p,q,p_1,p_2,\ldots,q_1,q_2,\ldots}$ is compact and has no triple of aligned points. Hence, $u_{pq}$ is a preserved extreme point of $\ball{\lipfree{X}}$ as we will prove in Theorem \ref{tm:pep_fk}. However, $\lambda^n-\lambda^{2n}<d(p_n,p)<\lambda^n+\lambda^{2n}$ and $\metex{p_n;p,q}\leq 2\dist(p_n,[p,q])<2\lambda^{2n}$, so it is simple to check that the hypotheses of Lemma \ref{lm:dpq_nonempty} are satisfied and this yields an element of $\setdpq{(p,q)}$ that lies over $p$. Similarly, the sequence $\set{q_n}$ yields an element of $\setdpq{(p,q)}$ that lies over $q$.

By removing e.g. the points $p_n$ for $n\geq 1$ from $X$, we obtain a similar example where $A_p$ is empty because $p$ is then isolated.
\end{example}

\section{Characterization of preserved extreme points}

We are finally ready to prove the characterization theorem for preserved extreme points of $\ball{\lipfree{X}}$:

\begin{theorem}
\label{tm:pep_fx}
Let $X$ be a pointed metric space, and let $p,q$ be distinct points of $X$. Then the following are equivalent:
\begin{enumerate}[\upshape (i)]
\item $u_{pq}$ is a preserved extreme point of $\ball{\lipfree{X}}$,
\item no point of $\beta X$ lies strictly between $p$ and $q$,
\item for every $\varepsilon>0$ there is $\delta>0$ such that $\metex{r;p,q}\geq\delta$ whenever $r\in X$ satisfies $d(p,r)\geq\varepsilon$ and $d(q,r)\geq\varepsilon$.
\end{enumerate}
\end{theorem}

\begin{proof}
The equivalence (ii)$\Leftrightarrow$(iii) is Lemma \ref{lm:pq_concave}, and the implication (i)$\Rightarrow$(ii) is Proposition \ref{pr:aligned_combination}(b). Only (ii)$\Rightarrow$(i) remains to be proved. Assume (ii), then Proposition \ref{pr:dpq_concave} implies that $\setdpq{(p,q)}=\set{(p,q),(q,p)}\cup A_p \cup A_q$ where all elements of $A_p$ and $A_q$ lie over $p$ and $q$, respectively.

Suppose that $u_{pq}=\lambda\dual{x}_1+(1-\lambda)\dual{x}_2$ for some $\lambda\in(0,1)$ and $\dual{x}_1,\dual{x}_2\in\ball{\dual{\Lip_0(X)}}$. By Lemma \ref{lm:measure_support}, $\dual{x}_1=\dual{\Phi}\mu_1$ and $\dual{x}_2=\dual{\Phi}\mu_2$ where $\mu_1,\mu_2\in\ball{M(\beta\widetilde{X})}$ are concentrated on $\setdpq{(p,q)}$. Hence, for $i=1,2$ we can write $\mu_i=a_i\delta_{(p,q)}+b_i\delta_{(q,p)}+\mu_i^{\prime}$ where $\mu_i^{\prime}$ is concentrated on $A_p\cup A_q$, so that $\dual{x}_i=(a_i-b_i)u_{pq}+\dual{\Phi}\mu_i^{\prime}$. Then, for any $f\in\Lip_0(X)$ we have
\begin{align}
\duality{u_{pq},f} &= \lambda\duality{(a_1-b_1)u_{pq}+\dual{\Phi}\mu_1^{\prime},f}+(1-\lambda)\duality{(a_2-b_2)u_{pq}+\dual{\Phi}\mu_2^{\prime},f} \notag \\
&= \big(\lambda(a_1-b_1)+(1-\lambda)(a_2-b_2)\big)\duality{u_{pq},f} \notag \\
&\quad + \lambda\int_{A_p\cup A_q}(\Phi f)\,d\mu_1^\prime + (1-\lambda)\int_{A_p\cup A_q}(\Phi f)\,d\mu_2^\prime \text{.} \label{eq:main_tm_eq}
\end{align}

Let $U$ and $V$ be neighborhoods of $p$ and $q$ with disjoint closures, and define $g\in\Lip(U\cup V)$ by $g=d(p,q)$ in $U$ and $g=0$ in $V$. Extend $g$ to all of $X$ using Proposition \ref{pr:metric_tietze}, and let $f=g-g(e)$. Then $f\in\Lip_0(X)$ and $\duality{u_{pq},f}=\Phi f(p,q)=1$. For every $\zeta\in A_p$ there is a net $\set{(x_i,y_i): i\in I}$ in $\widetilde{X}$ that converges to $\zeta$ in $\beta\widetilde{X}$ and such that $x_i,y_i$ are eventually in $U$, hence $\Phi f(x_i,y_i)=0$ eventually, and so $\Phi f(\zeta)=0$. Similarly, $\Phi f(\zeta)=0$ for $\zeta\in A_q$. Thus, for this particular $f$ the integrals in \eqref{eq:main_tm_eq} vanish and we get
\begin{align*}
1 &= \lambda(a_1-b_1)+(1-\lambda)(a_2-b_2) \\
&\leq \lambda(\abs{a_1}+\abs{b_1}+\norm{\mu_1^{\prime}})+(1-\lambda)(\abs{a_2}+\abs{b_2}+\norm{\mu_2^{\prime}}) \\
&= \lambda\norm{\mu_1}+(1-\lambda)\norm{\mu_2}\leq 1 \text{.}
\end{align*}
It follows that $\norm{\mu_1^\prime}=\norm{\mu_2^\prime}=0$, so $\dual{x}_1$ and $\dual{x}_2$ are multiples of $u_{pq}$. Thus $u_{pq}\in\Ext(\ball{\dual{\Lip_0(X)}})$.
\end{proof}

For compact $X$, Theorem \ref{tm:pep_fx} can be restated to involve (unpreserved) extreme points of $\ball{\lipfree{X}}$, too:

\begin{theorem}
\label{tm:pep_fk}
Let $X$ be a compact pointed metric space, and let $p,q$ be distinct points of $X$. Then the following are equivalent:
\begin{enumerate}[\upshape (i)]
\item $u_{pq}$ is a preserved extreme point of $\ball{\lipfree{X}}$,
\item $u_{pq}$ is an extreme point of $\ball{\lipfree{X}}$,
\item no point of $X$ lies strictly between $p$ and $q$.
\end{enumerate}
\end{theorem}

\begin{proof}
(i)$\Rightarrow$(ii) is trivial, (ii)$\Rightarrow$(iii) is Proposition \ref{pr:aligned_combination}(a), and (iii)$\Rightarrow$(i) is a consequence of Theorem \ref{tm:pep_fx} because $\beta X=X$.
\end{proof}

We remark that the hypothesis that $X$ is compact is essential in Theorem \ref{tm:pep_fk}. Simple counterexamples may be constructed using Theorem \ref{tm:pep_fx}. For instance, consider the subset $X$ of $c_0$ consisting of $e=0$, $p=2e_1$, and $q_n=e_1+(1+\frac{1}{n})e_n$ for $n\geq 2$, where $\set{e_n}$ is the canonical basis. Since $d(q_n,q_m)>1$ for different $n,m\geq 2$, the sequence $\set{q_n}$ has no cluster point in $X$, and $X$ is not compact. Also $\metex{q_n;p,e}=\frac{2}{n}$, so no point of $X$ lies strictly between $p$ and $e$. However, if $\xi$ is a cluster point of $q_n$ in $\beta X$, then $\metex{\xi;p,e}=0$, hence $u_{pe}$ is not a preserved extreme point. The recent preprint \cite{GPR_arxiv} presents a stronger example where there are no triples of aligned points and no preserved extreme points at all (see Remark 4.17).

\begin{definition}
We say that the pointed metric space $X$ is \emph{concave} if $u_{pq}$ is a preserved extreme point of $\ball{\lipfree{X}}$ for any distinct $p,q\in X$.
\end{definition}

In \cite[Open Problem in p. 53]{weaver}, N. Weaver conjectured that any compact metric space without triples of metrically aligned points is concave. As an immediate consequence of Theorem \ref{tm:pep_fk}, we obtain that the conjecture is actually a characterization of such spaces. We have recently learned that N. Weaver has independently found a proof of this fact \cite{weaver_pc}, which will appear in the second edition of \cite{weaver}.

\begin{corollary}
\label{co:concave}
Let $X$ be a compact pointed metric space. Then $X$ is concave if and only if no triple of distinct points of $X$ is metrically aligned.
\end{corollary}

Examples of concave spaces are \emph{H\"{o}lder spaces} $X^{\alpha}$, which are constructed from metric spaces $X$ by equipping them with the metric $d^{\alpha}$, where $\alpha\in (0,1)$. In \cite[Prop. 2.4.5]{weaver} they are shown to be concave in general. From Corollary \ref{co:concave}, we obtain an alternative proof that compact H\"{o}lder spaces are concave by noticing that for distinct $p,q,r\in X$ we have
\begin{equation*}
d(p,q)^{\alpha}\leq(d(p,r)+d(r,q))^{\alpha}<d(p,r)^{\alpha}+d(r,q)^{\alpha}
\end{equation*}
so that no set of three distinct points can be metrically aligned in $X^{\alpha}$. We remark that not all compact concave spaces are H\"{o}lder spaces, as shown by the following example:

\begin{example}
Consider decreasing sequences $\lambda_n\rightarrow 1$ and $a_n\rightarrow 0$, with $a_1<1$. Then $a_n^{\lambda_n}+(1-a_n)^{\lambda_n}<(a_n+1-a_n)^{\lambda_n}=1$, so we can choose positive $b_n\rightarrow 0$ such that
\begin{equation*}
(a_n^2+b_n^2)^{\lambda_n/2}+((1-a_n)^2+b_n^2)^{\lambda_n/2}<1 \text{.}
\end{equation*}
Note that the terms in parentheses are all smaller than 1. Let $X$ be the subset of $\ell^2$ consisting of $0$, $e_1$, and $r_n=a_ne_1+b_ne_n$ for $n\geq 2$, where $\set{e_n}$ is the canonical basis. Then $X$ is compact because $r_n\rightarrow 0$, and any triple of distinct points of $X$ spans an affine subspace of $\ell^2$ of dimension 2 so they cannot be metrically aligned because $\ell^2$ is strictly convex; hence $X$ is concave by Corollary \ref{co:concave}. However $X$ cannot be $\alpha$-H\"{o}lder for any $0<\alpha<1$: suppose there was a metric $d$ on $X$ such that $\norm{x-y}_2=d(x,y)^{\alpha}$ for any $x,y\in X$, and choose $n$ such that $\lambda_n<1/\alpha$. Then
\begin{align*}
d(0,r_n)+d(r_n,e_1) &= \norm{r_n}_2^{1/\alpha}+\norm{e_1-r_n}_2^{1/\alpha} \\
&= (a_n^2+b_n^2)^{1/2\alpha}+((1-a_n)^2+b_n^2)^{1/2\alpha} \\
&< (a_n^2+b_n^2)^{\lambda_n/2}+((1-a_n)^2+b_n^2)^{\lambda_n/2} \\ &<1=d(0,e_1)
\end{align*}
violating the triangle inequality.
\end{example}

\section{Open questions}

Theorem \ref{tm:pep_fx} provides a characterisation of preserved extreme points in Lipschitz-free spaces in terms of the geometry of the underlying metric space. In the recent preprint \cite{GPR_arxiv}, L. Garc\'ia, A. Proch\'azka and A. Rueda give a similar purely geometric characterisation for strongly exposed points. The authors say that a pair $(p,q)$ of distinct points of $X$ has \emph{property (Z)} if for every $\varepsilon>0$ there is $r\in X\setminus\set{p,q}$ such that
\begin{equation*}
\metex{r;p,q}\leq\varepsilon\min\set{d(p,r),d(q,r)}
\end{equation*}
and then prove the following:

\begin{theorem}
\label{tm:strongly_exposed}
If $X$ is a pointed metric space, then an element $u_{pq}$ is a strongly exposed point of $\ball{\lipfree{X}}$ if and only if $(p,q)\in\widetilde{X}$ does not have property (Z).
\end{theorem}

Notice that the condition in Lemma \ref{lm:dpq_nonempty} implies that the pair $(p,q)$ has property (Z); hence, the construction from Example \ref{ex:ap_aq} yields a preserved extreme point $u_{pq}$ that is not strongly exposed. One key difference between the conditions in Theorems \ref{tm:pep_fx} and \ref{tm:strongly_exposed} is the following: both involve the existence of nets $\set{r_i}$ such that $\metex{r_i;p,q}\rightarrow 0$, but property (Z) allows these nets to cluster at $p$ or $q$ whereas our condition explicitly prevents this.

Since Theorem \ref{tm:pep_fx} is essentially the converse of Proposition \ref{pr:aligned_combination}(b), one may ask whether Proposition \ref{pr:aligned_combination}(a) provides a similar geometric characterisation of extreme points in $\ball{\lipfree{X}}$:

\begin{question}
\label{q:ext_align}
Is $u_{pq}$ an extreme point of $\ball{\lipfree{X}}$ whenever no point of $X$ lies strictly between $p$ and $q$?
\end{question}

Theorem \ref{tm:pep_fk} shows that the answer to Question \ref{q:ext_align} is positive when $X$ is compact, but the general case remains unsolved.

Moreover, when $X$ is complete, all preserved extreme points are of the form $u_{pq}$, which strongly restricts their search (note that strongly exposed points are always preserved extreme \cite{GMZ14}), but we do not know whether the same restriction applies to extreme points in general:

\begin{question}
\label{q:ext_upq}
If $X$ is complete, are all extreme points of $\ball{\lipfree{X}}$ of the form $u_{pq}$?
\end{question}

The answer to Question \ref{q:ext_upq} is also known to be positive in some particular cases. Suppose $X$ is compact, and let $\lip_0(X)$ be the subspace of $\Lip_0(X)$ consisting of those functions $f$ satisfying the condition: for every $\varepsilon>0$ there is $\delta>0$ such that $\abs{\Phi f(p,q)}<\varepsilon$ whenever $d(p,q)<\delta$. We say that $\lip_0(X)$ \emph{separates points uniformly} if there is a constant $C\geq 1$ such that for any $p,q\in X$ there is $f\in\lip_0(X)$ with $L(f)\leq C$ and $\abs{f(p)-f(q)}=d(p,q)$. If this holds, then $\lipfree{X}$ is isometrically isomorphic to $\dual{\lip_0(X)}$ and Question \ref{q:ext_upq} has a positive answer for this $X$ \cite[Th. 3.3.3 and Cor. 3.3.6]{weaver}. Applying Theorem \ref{tm:pep_fk}, we summarize this as follows:

\begin{corollary}
\label{cr:compact_lip}
If $X$ is a compact pointed metric space such that $\lip_0(X)$ separates points uniformly, then the extreme points of $\ball{\lipfree{X}}$ are precisely the elements $u_{pq}$ where $p,q$ are distinct points of $X$ such that $d(p,q)<d(p,r)+d(q,r)$ for all $r\in X\setminus\set{p,q}$.
\end{corollary}

The condition in Corollary \ref{cr:compact_lip} is not satisfied in general (for instance, $\lip_0(X)$ may be trivial), but it is known to hold for compact H\"{o}lder spaces and for the Cantor ternary set \cite[Prop. 3.2.2]{weaver}. More recently, A. Dalet showed that it is also satisfied whenever the compact $X$ is countable \cite{dalet_1} or ultrametric \cite{dalet_2}.

\subsection*{Acknowledgements}
The research of the second author was partially supported by MINECO grant MTM2014-57838-C2-1-P and Fundaci\'on S\'eneca, Regi\'on de Murcia  grant 19368/PI/14.


\end{document}